\documentclass{amsart}
\usepackage{eurosym}
\usepackage{amssymb}
\usepackage{amsmath}
\usepackage{amsfonts}

\newtheorem{theorem}{Theorem}
\theoremstyle{plain}
\newtheorem{acknowledgement}{Acknowledgement}

\newtheorem{definition}{Definition}

\newtheorem{lemma}{Lemma}

\numberwithin{equation}{section}

\begin{document}
\title[Kirchhoff elliptic systems with variable parameters]{Existence result
for Kirchhoff elliptic system with variable parameters}
\dedicatory{.}

\begin{abstract}
The paper deals with the study of the existence result for a Kirchhoff
elliptic system with additive right hand side and variable parameters by
using the sub-super solutions method. Our study is the second result of our previous once in (Math. Methods Appl. Sci. 41 (2018),
5203-5210).
\end{abstract}

\author{Mohamed Haiour$^{1}$, Youcef Bouizem$^{4}$, Salah Boulaaras$^{2,3}$,
and Rafik Guefaifia$^{\text{3}}$}
\address{$^{\text{1}}$Department of Mathematics, Faculty of Sciences,
University of Annaba 12002, Algeria\\
$^{\text{2}}$Department of Mathematics, College of Sciences and Arts,
Al-Rass,\\
Qassim University, Kingdom of Saudi Arabia. \\
$^{\text{3}}$Laboratory of Fundamental and Applied Mathematics of Oran
(LMFAO), University of Oran 1, Ahmed Benbella. Algeria.\\
$^{\text{4}}$Department of Mathematics, Faculty of Mathematics and
Informatics, USTOMB, Oran, Algeria\\
}
\thanks{}
\maketitle

\section{Introduction}

Consider the following system
\begin{equation}
\bigskip \left\{
\begin{array}{l}
-A\left( \int\limits_{\Omega }\left\vert \nabla u\right\vert ^{2}dx\right)
\triangle u= \alpha \left( x\right) f\left( v\right) +\beta \left( x\right)
g\left( u\right) \text{ in }\Omega , \\
\\
-B\left( \int\limits_{\Omega }\left\vert \nabla v\right\vert ^{2}dx\right)
\triangle v=\gamma \left( x\right) h\left( u\right) +\eta \left( x\right) l
\left( v\right) \text{ in }\Omega , \\
\\
u=v=0\text{ on }\partial \Omega ,%
\end{array}%
\right.  \label{1.1}
\end{equation}%
where $\Omega \subset
\mathbb{R}
^{N}$ \ $\left( N\geq 3\right) $ is a bounded smooth domain with $C^{2}$
boundary $\partial \Omega ,$ and $A,$ $B$ :$%
\mathbb{R}
^{+}\rightarrow
\mathbb{R}
^{+}$ are continuous functions with further conditions to be given later, $%
\alpha ,\beta ,\gamma ,\eta \in C\left( \overline{\Omega }\right)$.

This nonlocal problem originates from the stationary version of Kirchhoff's
work \cite{7} in 1883
\begin{equation}
\rho \frac{\partial ^{2}u}{\partial t^{2}}-\left( \frac{P_{0}}{h}+\frac{E}{2L%
}\int\limits_{0}^{L}\left\vert \frac{\partial u}{\partial x}\right\vert
^{2}dx\right) \frac{\partial ^{2}u}{\partial x^{2}}=0,  \label{1.2}
\end{equation}

where Kirchhoff extended the classical d'Alembert's wave equation by
considering the effect of the changes in the length of the string during
vibrations. The parameters in (\ref{1.2}) have the following meanings: $L$
is the length of the string, $h$ is the area of the cross-section, $E$ \ is
the Young modulus of the material, $\rho $ is the mass density, and $P_{0}$
is the initial tension.

Recently, The problems associated to Laplacian operator and Kirchhoff
elliptic equations have been heavily studied, we refer to \cite{1}, \cite{X}%
, \cite{2}, \cite{3}, \cite{log}, \cite{Y}, \cite{Z}, \cite{W}, \cite{11}
and \cite{12}.

In \cite{1}, Alves and Correa proved the validity of Sub-super solutions
method for problems of Kirchhoff class involving a single equation and a
boundary condition
\begin{equation*}
\left\{
\begin{array}{l}
-M\left( \Vert u\Vert ^{2}\right) \Delta u=f\left( x,u\right) \text{ in }%
\Omega  \\
u=0\text{ on }\partial \Omega ,%
\end{array}%
\right.
\end{equation*}%
with $f\in C\left( \overline{\Omega }\times \mathbb{R}\right) $. \newline
By using a comparison principle that requires $M$ to be non-negative and
non-increasing in $\left[ 0,+\infty \right) $, with $H\left( t\right)
:=M\left( t^{2}\right) t$ increasing and $H\left( \mathbb{R}\right) =\mathbb{%
R}$, they managed to prove the existence of positive solutions assuming $f$
increasing in the variable $u$ for each $x\in \Omega $ fixed.

For systems involving similar class of equations, this result can not be
used directly, i.e. the existence of a subsolution and a supersolution does
not guarantee the existence of the solution. Therefore, a further
construction is needed. As in \cite{Original}, where we studied the system
\begin{equation}
\bigskip \left\{
\begin{array}{l}
-A\left( \int\limits_{\Omega }\left\vert \nabla u\right\vert ^{2}dx\right)
\triangle u= \lambda_1 f\left( v\right) +\mu_1 g\left( u\right) \text{ in }%
\Omega , \\
\\
-B\left( \int\limits_{\Omega }\left\vert \nabla v\right\vert ^{2}dx\right)
\triangle v=\lambda_2 h\left( u\right) +\mu_2 \left( x\right) l \left(
v\right) \text{ in }\Omega , \\
\\
u=v=0\text{ on }\partial \Omega ,%
\end{array}%
\right.
\end{equation}

Using a weak positive supersolution as first term of a constructed iterative
sequence $\left(u_n, v_n\right)$ in $H^1_0\left(\Omega\right)\times
H^1_0\left(\Omega\right)$, and a comparison principle introduced in \cite{1}%
, the authors established the convergence of this sequence to a positive
weak solution of the considered problem.

In this paper, we generalize the previous work in \cite{Original} by
considering variable parameters $\alpha, \beta, \gamma$ and $\eta$ in the
right hand side of (\ref{1.1}). We also give a better subsolution providing
easier computations.

\section{Existence result}

\begin{definition}
$\left( u,v\right) \in \left( H_{0}^{1}\left( \Omega \right) \times
H_{0}^{1}\left( \Omega \right) \right) ,$ is called a weak solution of (\ref%
{1.1}) if it satisfies
\end{definition}

\begin{equation*}
\begin{array}{c}
A\left( \int\limits_{\Omega }\left\vert \nabla u\right\vert ^{2}dx\right)
\int\limits_{\Omega }\nabla u\nabla \phi dx=\int\limits_{\Omega }\alpha
\left( x\right) f\left( v\right) \phi dx+\int\limits_{\Omega }\beta \left(
x\right)g\left( u\right) \phi \text{ }dx\text{ in }\Omega , \\
\\
B\left( \int\limits_{\Omega }\left\vert \nabla v\right\vert ^{2}dx\right)
\int\limits_{\Omega }\nabla v\nabla \psi dx=\int\limits_{\Omega }\gamma
\left( x\right) h\left( u\right) \psi dx+\int\limits_{\Omega }\eta \left(
x\right) l \left( v\right) \psi \text{ }dx\text{ in }\Omega%
\end{array}%
\end{equation*}

for all $\left( \phi ,\psi \right) \in \left( H_{0}^{1}\left( \Omega \right)
\times H_{0}^{1}\left( \Omega \right) \right) .$

\begin{definition}
Let $\left( \underline{u},\underline{v}\right) ,\left( \overline{u},%
\overline{v}\right) $ be a pair of nonnegative functions in $\left(
H_{0}^{1}\left( \Omega \right) \times H_{0}^{1}\left( \Omega \right) \right)
$, they are called positive weak subsolution and positive weak supersolution
(respectively) of (\ref{1.1}) if \ they satisfy the following
\end{definition}

\begin{equation*}
\begin{array}{c}
A\left( \int\limits_{\Omega }\left\vert \nabla \underline{u}\right\vert
^{2}dx\right) \int\limits_{\Omega }\nabla \underline{u}\nabla \phi dx\leq
\int\limits_{\Omega }\alpha \left( x\right) f\left( \underline{v}\right)
\phi \text{ }dx+\int\limits_{\Omega }\beta \left( x\right) g\left(
\underline{u}\right) \phi \text{ }dx \\
\\
B\left( \int\limits_{\Omega }\left\vert \nabla \underline{v}\right\vert
^{2}dx\right) \int\limits_{\Omega }\nabla \underline{v}\nabla \psi dx\leq
\int\limits_{\Omega }\gamma \left( x\right) h\left( \underline{u}\right)
\psi \text{ }dx+\text{ }\int\limits_{\Omega }\eta \left( x\right) l \left(
\underline{v}\right) \psi \text{ }dx%
\end{array}%
\end{equation*}

and

\begin{equation*}
\begin{array}{c}
A\left( \int\limits_{\Omega }\left\vert \nabla \overline{u}\right\vert
^{2}dx\right) \int\limits_{\Omega }\nabla \overline{u}\nabla \phi dx\geq
\int\limits_{\Omega }\alpha \left( x\right) f\left( \overline{v}\right) \phi
\text{ }dx+\int\limits_{\Omega }\beta \left( x\right) g\left( \overline{u}%
\right) \phi \text{ }dx \\
\\
B\left( \int\limits_{\Omega }\left\vert \nabla \overline{v}\right\vert
^{2}dx\right) \int\limits_{\Omega }\nabla \overline{v}\nabla \psi dx\geq
\int\limits_{\Omega }\gamma \left( x\right) h\left( \overline{u} \right)
\psi \text{ }dx\text{ }+\int\limits_{\Omega }\eta \left( x\right) l \left(
\overline{v}\right) \psi \text{ }dx%
\end{array}%
\end{equation*}
for all $\left( \phi ,\psi \right) \in \left( H_{0}^{1}\left( \Omega \right)
\times H_{0}^{1}\left( \Omega \right) \right)$, with $\phi \geq 0$ and $\psi
\geq 0$,\newline

and $\left(\underline{u},\underline{v}\right) ,\left( \overline{u},\overline{%
v}\right) = $ $\left( 0,0\right) $ on $\partial \Omega $.

\begin{lemma}
(Comparison principle \cite{1} ) Let $M:\mathbb{R}^{+}\rightarrow \mathbb{R}%
^{+}$ be a continuous nonincreasing function such that
\begin{equation}
M\left( s\right) >m_{0} > 0,\text{ for all }s\geq s_{0},  \label{2.1}
\end{equation}%
and $H\left( t\right) =tM\left( t^{2}\right)$ \ increasing on $\mathbb{R}%
^{+} $ .\newline
If $u_1,u_2$ are two non-negative functions verifying
\begin{equation}
\left\{
\begin{array}{l}
-M\left( \int\limits_{\Omega }\left\vert \nabla u_1\right\vert ^{2}dx\right)
\triangle u_1\geq -M\left( \int\limits_{\Omega }\left\vert \nabla
u_2\right\vert ^{2}dx\right) \triangle u_2\text{ in }\Omega , \\
\\
u=v=0\text{ on }\partial \Omega ,%
\end{array}%
\right.
\end{equation}

then $u_1\geq u_2$ $a.e.$ in $\Omega .$\newline
\label{2.2}
\end{lemma}

Before stating and proving our main result, here are the conditions we need.
\newline

$\left( H1\right) $ $A,B: \mathbb{R}^{+}\rightarrow\mathbb{R}^{+}$ are two
continuous and increasing functions that satisfy the monotonicity conditions
of lemma \ref{2.2} so that we can use the Comparison principle, and assume
further that there exists $a_{i},b_{i} > 0, \ \ i=1,2,$

\begin{equation*}
a_{1}\leq A\left( t\right) \leq a_{2},\text{ \ \ \ }b_{1}\leq B\left(
t\right) \leq b_{2}\text{ \ for all }t\in
\mathbb{R}
^{+}.
\end{equation*}

$\left( H2\right) $ $\alpha ,\beta ,\gamma ,\eta \in C\left( \overline{%
\Omega }\right) $ and
\begin{equation*}
\alpha \left( x\right) \geq \alpha _{0}>0, \ \beta \left( x\right) \geq
\beta _{0}>0, \ \gamma \left( x\right) \geq \gamma _{0}>0, \ \eta \left(
x\right) \geq \eta _{0}>0
\end{equation*}
for all $x\in \Omega .$

$\left( H3\right) $ $f,$ $g$, $h,$ and $l $ are continuous on $\left[
0,+\infty \right[ ,$ $C^{1}$ on $\left( 0,+\infty \right) ,$ and increasing
functions of infinite growth

\begin{equation*}
\lim_{t\rightarrow +\infty }f\left( t\right) =+\infty , \lim_{t\rightarrow
+\infty }l \left( t\right) =+\infty, \lim_{t\rightarrow +\infty }g\left(
t\right) =+\infty ,\text{ }\lim_{t\rightarrow +\infty }h\left( t\right)
=+\infty.
\end{equation*}

$\left( H4\right) $ For all $K> 0$
\begin{equation*}
\lim_{t\rightarrow +\infty }\frac{f\left( K\left( h\left( t\right) \right)
\right) }{t}=0.
\end{equation*}

$\left( H5\right) $%
\begin{equation*}
\lim_{t\rightarrow +\infty }\frac{g\left( t\right) }{t}=\lim_{t\rightarrow
+\infty }\frac{l \left( t\right) }{t}=0.
\end{equation*}

\begin{theorem}
For large values of $\alpha _{0}+\beta _{0}$ and $\gamma _{0}+\eta _{0}$,
system (\ref{1.1}) admits a large positive weak solution if conditions $%
\left( H1\right) -\left( H5\right) $ are satisfied.
\end{theorem}

\begin{proof}[Proof of Theorem 1]
Consider $\sigma $ the first eigenvalue of $\ -\triangle $ with Dirichlet
boundary conditions and $\phi _{1}$ the corresponding positive eigenfunction
with $\left\Vert \phi _{1}\right\Vert =1$ and $\phi _{1} \in C^{\infty}\left(%
\overline{\Omega}\right)$ (see \cite{Evans}).\newline

Let $S= \sup\limits_{x \in \Omega}\{\sigma \phi_1^{2}- |\nabla \phi_1|^{2}
\} $, then from growth conditions (H3)
\begin{equation*}
f\left(t\right) \geq S, \ g\left(t\right) \geq S, \ h\left(t\right) \geq S,
\ l\left(t\right) \geq S, \ \ \ \text{for } t \text { large enough}.
\end{equation*}

For each $\alpha _{0}+ \beta _{0}$ and $\gamma _{0}+ \eta _{0}$ large, let
us define
\begin{equation*}
\underline{u}= \frac{\alpha _{0}+ \beta_{0}}{2a_{2}}\phi _{1}^{2}
\end{equation*}%
and
\begin{equation*}
\underline{v}= \frac{ \gamma _{0}+\eta_{0}}{2b_{2}} \phi _{1}^{2},
\end{equation*}

where $a_{2},b_{2}$ are given by condition $\left( H1\right) .$ Let us show
that $\left( \underline{u},\underline{v}\right) $ is a subsolution of
problem (\ref{1.1}) for $\alpha _{0}+\beta _{0}$ and $\gamma _{0}+\eta _{0}$
large enough. Indeed, let $\phi \in H_{0}^{1}\left( \Omega \right) $ with $%
\phi \geq 0$ in $\Omega .$ By $\left( H1\right) -$ $\left( H3\right) ,$ we
get

\begin{eqnarray*}
A\left( \int\limits_{\Omega}\left\vert \nabla \underline{u}\right\vert
^{2}dx\right) \int\limits_{\Omega}\nabla \underline{u}.\nabla \phi dx
&=&A\left( \int\limits_{\Omega}\left\vert \nabla \underline{u}\right\vert
^{2}dx\right) \frac{ \alpha _{0}+\beta _{0} }{a_{2}}\int\limits_{\Omega}\phi
_{1}\nabla \phi _{1}.\nabla \phi dx \\
&& \\
&=&\frac{ \alpha _{0}+\beta _{0}}{a_{2}}A\left( \int\limits_{\Omega
}\left\vert \nabla \underline{u}\right\vert ^{2}dx\right) \times \\
&&\left\{ \int\limits_{\Omega }\nabla \phi _{1}\nabla \left( \phi _{1}.\phi
\right) dx-\int\limits_{\Omega}\left\vert \nabla \phi _{1}\right\vert
^{2}\phi dx\right\} \\
&& \\
&=&\frac{\alpha _{0}+\beta _{0}}{a_{2}}A\left( \int\limits_{\Omega
}\left\vert \nabla \underline{u}\right\vert ^{2}dx\right)
\int\limits_{\Omega }\left( \sigma \phi _{1}^{2}-\left\vert \nabla \phi
_{1}\right\vert ^{2}\right) \phi dx. \\
&& \\
&\leq & \left(\alpha _{0}+\beta _{0}\right) \int\limits_{\Omega }S \phi dx.
\\
&& \\
&\leq & \int\limits_{\Omega } \alpha\left(x\right)f \left(\underline{v}%
\right)\phi dx + \int\limits_{\Omega } \beta\left(x\right)g \left(\underline{%
u}\right)\phi dx
\end{eqnarray*}

for $\alpha _{0}+\beta _{0}>0$ large enough, and all $\phi \in
H_{0}^{1}\left( \Omega \right) $ with $\phi \geq 0$ in $\Omega .$

Similarly,

\begin{equation*}
B\left( \int\limits_{\Omega }\left\vert \nabla \underline{v}\right\vert
^{2}dx\right) \int\limits_{\Omega }\nabla \underline{v}\nabla \psi dx\leq
\int\limits_{\Omega }\gamma \left( x\right) h\left( \underline{u}\right)
\psi dx\text{ }+\int\limits_{\Omega }\eta \left( x\right) \l \left(
\underline{v}\right) \psi dx\text{ in }\Omega
\end{equation*}

for $\gamma _{0}+\eta _{0}>0$ large enough and all $\psi \in H_{0}^{1}\left(
\Omega \right) $ with $\psi \geq 0$ in $\Omega$.\newline

Also notice that $\underline{u}>0$ and $\underline{v}>0$ in $\Omega $,
\newline
$\underline{u}\rightarrow +\infty $ and $\underline{v}\rightarrow +\infty $
as $\alpha _{0}+\beta _{0}\rightarrow +\infty $ and $\gamma _{0}+\eta
_{0}\rightarrow +\infty $ .\newline

For the supersolution part, consider $e $ the solution of the following
problem

\begin{equation}
\left\{
\begin{array}{c}
-\triangle e=1\text{ in }\Omega , \\
\\
e=0\text{ on }\partial \Omega .%
\end{array}%
\right.  \label{2.10}
\end{equation}

We give the supersolution of problem (\ref{1.1}) by
\begin{equation*}
\overline{u}=Ce,\ \ \text{ }\overline{v}= \left(\Vert \gamma \Vert{_\infty }%
+ \Vert \eta \Vert_{\infty }\right) h\left( C\left\Vert e\right\Vert
_{\infty }\right) e,
\end{equation*}

where $C>0$ is a large positive real number to be given later. \newline
Indeed, for all $\phi \in H_{0}^{1}\left( \Omega \right) $ with $\phi \geq 0$
in $\Omega$, we get from (\ref{2.10}) and the condition $\left( H1\right) $

\begin{eqnarray*}
A\left( \int\limits_{\Omega }\left\vert \nabla \overline{u}\right\vert
^{2}dx\right) \int\limits_{\Omega }\nabla \overline{u}.\nabla \phi dx
&=&A\left( \int\limits_{\Omega }\left\vert \nabla \overline{u}\right\vert
^{2}dx\right) C\int\limits_{\Omega }\nabla e .\nabla \phi dx \\
&& \\
&=&A\left( \int\limits_{\Omega }\left\vert \nabla \overline{u}\right\vert
^{2}dx\right) C\int\limits_{\Omega }\phi dx \\
&& \\
&\geq &a_{1}C\int\limits_{\Omega }\phi dx,
\end{eqnarray*}

By $\left( H4\right) $ and $\left( H5\right) $, we can choose $C$ large
enough so that

\begin{equation*}
a_{1}C\geq \left\Vert \alpha \right\Vert _{\infty }f\left[ \left(\Vert
\gamma \Vert{_\infty }+ \Vert \eta \Vert_{\infty }\right) h\left(C\left\Vert
e\right\Vert _{\infty }\right) \left\Vert e\right\Vert _{\infty }\right]
+\left\Vert \beta \right\Vert _{\infty }g\left( C\left\Vert e\right\Vert
_{\infty }\right) .
\end{equation*}

Therefore,

\begin{equation}
\begin{array}{l}
A\left( \int\limits_{\Omega }\left\vert \nabla \overline{u}\right\vert
^{2}dx\right) \int\limits_{\Omega }\nabla \overline{u}.\nabla \phi dx \\
\\
\geq \left[ \left\Vert \alpha \right\Vert _{\infty }f\left[ \left(\Vert
\gamma \Vert{_\infty }+ \Vert \eta \Vert_{\infty }\right) h\left(C\left\Vert
e\right\Vert _{\infty }\right) \left\Vert e\right\Vert _{\infty }\right]
+\left\Vert \beta \right\Vert _{\infty }g\left( C\left\Vert e\right\Vert
_{\infty }\right) \right] \int\limits_{\Omega }\phi dx \\
\\
\geq \left\Vert \alpha \right\Vert _{\infty }\int\limits_{\Omega }f\left[
\left(\Vert \gamma \Vert{_\infty }+ \Vert \eta \Vert_{\infty }\right)
h\left(C\left\Vert e\right\Vert _{\infty }\right) \left\Vert e\right\Vert
_{\infty }\right] \phi dx+ \left\Vert \beta \right\Vert _{\infty }
\int\limits_{\Omega }g\left( C\left\Vert e\right\Vert _{\infty }\right) \phi
dx \\
\\
\geq\int\limits_{\Omega }\alpha \left( x\right) f\left( \overline{v}\right)
\phi dx+\int\limits_{\Omega }\beta \left( x\right) g\left( \overline{u}%
\right) \phi dx.%
\end{array}
\label{2.11}
\end{equation}

Also,%
\begin{equation}
\begin{array}{l}
B\left( \int\limits_{\Omega }\left\vert \nabla \overline{v}\right\vert
^{2}dx\right) \int\limits_{\Omega }\nabla \overline{v}\nabla \psi dx =
\left( \left\Vert \gamma \right\Vert _{\infty }+\left\Vert \eta \right\Vert
_{\infty }\right) \int\limits_{\Omega }h\left( C\left\Vert e\right\Vert
_{\infty }\right) \psi dx \\
\\
\hspace{4 cm} \geq \int\limits_{\Omega }\gamma \left( x\right) h\left(
\overline{u}\right) \psi dx+\int\limits_{\Omega }\eta \left( x\right)
h\left( C\left\Vert e\right\Vert _{\infty }\right) \psi dx.%
\end{array}
\label{2.12}
\end{equation}

Using $\left( H4\right) $ and $\left( H5\right) $ again for $C$ large enough
we get

\begin{equation}
h\left( C\left\Vert e\right\Vert _{\infty }\right) \geq l \left[ \left(
\left\Vert \gamma \right\Vert _{\infty }+\left\Vert \eta \right\Vert
_{\infty }\right) h\left( C\left\Vert e\right\Vert _{\infty }\right)
\left\Vert e\right\Vert _{\infty }\right] \geq l \left( \overline{v}\right) .
\label{2.13}
\end{equation}

Combining (\ref{2.12}) and (\ref{2.13}), we obtain
\begin{equation}
B\left( \int\limits_{\Omega }\left\vert \nabla \overline{v}\right\vert
^{2}dx\right) \int\limits_{\Omega }\nabla \overline{v}\nabla \psi dx\geq
\int\limits_{\Omega }\gamma \left( x\right) h\left( \overline{u}\right) \psi
dx+\int\limits_{\Omega }\eta \left( x\right) l \left( \overline{v}\right)
\psi dx.  \label{2.14}
\end{equation}

By (\ref{2.11}) and (\ref{2.14}) we conclude that $\left( \overline{u},%
\overline{v}\right) $ is a supersolution of problem (\ref{1.1}).\newline
Furthermore, \underline{$u$} $\leq \overline{u}$ and \underline{$v$} $\leq
\overline{v}$ for $C$ \ chosen large enough.\newline

Now, we use a similar argument to \cite{Original} in order to obtain a weak
solution of our problem. Consider the following sequence $\left\{ \left(
u_{n},v_{n}\right) \right\} \subset \left( H_{0}^{1}\left( \Omega \right)
\times H_{0}^{1}\left( \Omega \right) \right) $ where: $u_{0}:=\overline{u}%
,v_{0}=\overline{v}$ and $\left( u_{n},v_{n}\right) $ is the unique solution
of

\begin{equation}
\left\{
\begin{array}{l}
-A\left( \int\limits_{\Omega }\left\vert \nabla u_{n}\right\vert
^{2}dx\right) \triangle u_{n}=\alpha \left( x\right) f\left( v_{n-1}\right)
+\beta \left( x\right) g\left( u_{n-1}\right) \text{ in }\Omega , \\
\\
-B\left( \int\limits_{\Omega }\left\vert \nabla v_{n}\right\vert
^{2}dx\right) \triangle v_{n}=\gamma \left( x\right) h\left( u_{n-1}\right)
+\eta \left( x\right) l \left( v_{n-1}\right) \text{ in }\Omega , \\
\\
u_{n}=v_{n}=0\text{ on }\partial \Omega .%
\end{array}%
\right.  \label{2.15}
\end{equation}

Since $A$ and $B$ satisfy $\left(H1\right)$ and $\alpha \left(x\right)
f\left( v_{n-1}\right) ,$ $\beta \left(x\right)g\left( u_{n-1}\right) ,
\gamma \left(x\right)h\left( u_{n-1}\right) ,$ and $\eta
\left(x\right)l\left( v_{n-1}\right) \in L^{2}\left( \Omega \right) $ $%
\left( \text{in }x\right) ,$

we deduce from a result in \cite{1} that system (\ref{2.15}) has a unique
solution $\left( u_{n},v_{n}\right) \in \left( H_{0}^{1}\left( \Omega
\right) \times H_{0}^{1}\left( \Omega \right) \right) .$

Using (\ref{2.15}) and the fact that $\left( u_{0},v_{0}\right) $ is a
supersolution of (\ref{1.1}), we get

\begin{equation*}
\left\{
\begin{array}{c}
-A\left( \int\limits_{\Omega }\left\vert \nabla u_{0}\right\vert
^{2}dx\right) \triangle u_{0}\geq \alpha \left( x\right) f\left(
v_{0}\right) +\text{ }\beta \left( x\right) g\left( u_{0}\right) =-A\left(
\int\limits_{\Omega }\left\vert \nabla u_{1}\right\vert ^{2}dx\right)
\triangle u_{1}, \\
\\
-B\left( \int\limits_{\Omega }\left\vert \nabla v_{0}\right\vert
^{2}dx\right) \triangle v_{0}\geq \gamma \left( x\right) h\left(
u_{0}\right) +\eta \left( x\right) l \left( v_{0}\right) \text{ }=-B\left(
\int\limits_{\Omega }\left\vert \nabla v_{1}\right\vert dx\right) \triangle
v_{1}%
\end{array}%
\right.
\end{equation*}

then by Lemma 1$,$ $u_{0}\geq u_{1}$ and $v_{0}\geq v_{1}.$ Also, since $%
u_{0}\geq $ $\underline{u}$, $v_{0}\geq $ $\underline{v}$ and the
monotonicity of $\ f,$ $g,$ $h,$ and $l $ one has

\begin{eqnarray*}
-A\left( \int\limits_{\Omega }\left\vert \nabla u_{1}\right\vert
^{2}dx\right) \triangle u_{1} &=&\alpha \left( x\right) f\left( v_{0}\right)
+\text{ }\beta \left( x\right) g\left( u_{0}\right) \\
&& \\
&\geq &\alpha \left( x\right) f\left( \underline{v}\right) +\text{ }\beta
\left( x\right) g\left( \underline{u}\right) \geq -A\left(
\int\limits_{\Omega }\left\vert \nabla \underline{u}\right\vert
^{2}dx\right) \triangle \underline{u}, \\
&& \\
-B\left( \int\limits_{\Omega }\left\vert \nabla v_{1}\right\vert
^{2}dx\right) \triangle v_{1} &=& \gamma \left( x\right) h\left(
u_{0}\right) + \eta \left( x\right) l \left( v_{0}\right) \\
&\geq & \gamma \left( x\right) h\left( \underline{u}\right) +\eta \left(
x\right) l \left( \underline{v}\right) \geq -B\left( \int\limits_{\Omega
}\left\vert \nabla \underline{v}\right\vert ^{2}dx\right) \triangle
\underline{v}
\end{eqnarray*}
according to Lemma 1 again, we obtain $u_{1}\geq \underline{u},$ $v_{1}\geq
\underline{v}.$ \newline

Repeating the same argument for $u_{2},v_{2}$, observe that
\begin{eqnarray*}
-A\left( \int\limits_{\Omega }\left\vert \nabla u_{1}\right\vert
^{2}dx\right) \triangle u_{1} &=& \alpha \left( x\right) f\left(
v_{0}\right) + \beta \left( x\right) g\left( u_{0}\right) \\
&\geq & \alpha \left( x\right) f\left( v_{1}\right) + \beta \left( x\right)
g\left( u_{1}\right) =-A\left( \int\limits_{\Omega }\left\vert \nabla
u_{2}\right\vert ^{2}dx\right) \triangle u_{2}, \\
-B\left( \int\limits_{\Omega }\left\vert \nabla v_{1}\right\vert dx\right)
\triangle v_{1} &=& \gamma \left( x\right) h\left( u_{0}\right) + \eta
\left( x\right) l \left( v_{0}\right) \\
&\geq & \gamma \left( x\right) h\left( u_{1}\right) + \eta \left( x\right) l
\left( v_{1}\right) =-B\left( \int\limits_{\Omega }\left\vert \nabla
v_{2}\right\vert ^{2}dx\right) \triangle v_{2},
\end{eqnarray*}

then $u_{1}\geq u_{2},$ $v_{1}\geq v_{2}.$ Similarly, we get $u_{2}\geq
\underline{u}$ and $v_{2}\geq \underline{v}$ from

\begin{eqnarray*}
-A\left( \int\limits_{\Omega }\left\vert \nabla u_{2}\right\vert
^{2}dx\right) \triangle u_{2} &=& \alpha \left( x\right) f\left(
v_{1}\right) + \beta \left( x\right) g\left( u_{1}\right) \\
&\geq & \alpha \left( x\right) f\left( \underline{v}\right) + \beta \left(
x\right) g\left( \underline{u}\right) \geq -A\left( \int\limits_{\Omega
}\left\vert \nabla \underline{u}\right\vert ^{2}dx\right) \triangle
\underline{u}, \\
&& \\
-B\left( \int\limits_{\Omega }\left\vert \nabla v_{2}\right\vert
^{2}dx\right) \triangle v_{2} &=& \gamma \left( x\right) h\left(
u_{1}\right) + \eta \left( x\right) l \left( v_{1}\right) \\
&\geq & \gamma \left( x\right) h\left( \underline{u}\right) + \eta \left(
x\right) l \left( \underline{v}\right) \geq -B\left( \int\limits_{\Omega
}\left\vert \nabla \underline{v}\right\vert ^{2}dx\right) \triangle
\underline{v}.
\end{eqnarray*}

By repeating these implementations we construct a bounded decreasing
sequence $\left\{ \left( u_{n},v_{n}\right) \right\} \subset \left(
H_{0}^{1}\left( \Omega \right) \times H_{0}^{1}\left( \Omega \right) \right)
$ verifying
\begin{equation}
\overline{u}=u_{0}\geq u_{1}\geq u_{2}\geq ...\geq u_{n}\geq ...\geq
\underline{u}>0,  \label{2.16}
\end{equation}

\begin{equation}
\overline{v}=v_{0}\geq v_{1}\geq v_{2}\geq ...\geq v_{n}\geq ...\geq
\underline{v}>0.  \label{2.17}
\end{equation}

By continuity of functions $f,g,$ $h,$ and $l $ and the definition of the
sequences $\left( u_{n}\right) $ and $\left( v_{n}\right) ,$ there exist
positive constants $C_{i}>0,$ $i=1,...,4$ such that%
\begin{equation}
\left\vert f\left( v_{n-1}\right) \right\vert \leq C_{1},\ \text{\ \ }%
\left\vert g\left( u_{n-1}\right) \right\vert \text{ }\leq C_{2},\
\left\vert h\left( u_{n-1}\right) \right\vert \leq C_{3}  \label{2.18}
\end{equation}

and
\begin{equation*}
\left\vert l \left( u_{n-1}\right) \right\vert \text{\ }\leq C_{4}\text{ for
all }n.
\end{equation*}

From (\ref{2.18}), multiplying the first equation of (\ref{2.15}) by $u_{n}$%
, integrating, using Holder inequality and Sobolev embedding we check that

\begin{eqnarray*}
a_{1}\int\limits_{\Omega }\left\vert \nabla u_{n}\right\vert ^{2}dx &\leq
&A\left( \int\limits_{\Omega }\left\vert \nabla u_{n}\right\vert
^{2}dx\right) \int\limits_{\Omega }\left\vert \nabla u_{n}\right\vert ^{2}dx
\\
&& \\
&=& \int\limits_{\Omega }\alpha \left( x\right) f\left( v_{n-1}\right)
u_{n}dx+ \int\limits_{\Omega }\beta \left( x\right) g\left( u_{n-1}\right)
u_{n}dx \\
&& \\
&\leq & \left\Vert \alpha \right\Vert _{\infty }\int\limits_{\Omega
}\left\vert f\left( v_{n-1}\right) \right\vert \text{ }\left\vert
u_{n}\right\vert dx+ \left\Vert \beta \right\Vert _{\infty
}\int\limits_{\Omega }\left\vert g\left( u_{n-1}\right) \right\vert \text{ }%
\left\vert u_{n}\right\vert dx \\
&& \\
&\leq &C_{1} \int\limits_{\Omega }\text{ }\left\vert u_{n}\right\vert
dx+C_{2} \int\limits_{\Omega }\left\vert u_{n}\right\vert dx \\
&& \\
&\leq &C_{5}\left\Vert u_{n}\right\Vert _{H_{0}^{1}\left( \Omega \right) }
\end{eqnarray*}

or
\begin{equation}
\left\Vert u_{n}\right\Vert _{H_{0}^{1}\left( \Omega \right) }\leq C_{5},\
\forall n,  \label{2.19}
\end{equation}

where $C_{5}>0$ is a constant independent of $n.$ Similarly, there exist $%
C_{6}>0$ independent of $n$ such that%
\begin{equation}
\left\Vert v_{n}\right\Vert _{H_{0}^{1}\left( \Omega \right) }\leq C_{6},\ \
\forall n.  \label{2.20}
\end{equation}

From (\ref{2.19}) and (\ref{2.20}), we deduce that $\left\{ \left(
u_{n},v_{n}\right) \right\} $ admits a weakly converging subsequence in $%
H_{0}^{1}\left( \Omega ,%
\mathbb{R}
^{2}\right) $ to a limit $\left( u,v\right) $ satisfying $u\geq $ $%
\underline{u}$ $>0$ and $v\geq $ \underline{$v$} $>0.$ Being monotone and
also using a standard regularity argument, $\left\{ \left(
u_{n},v_{n}\right) \right\} $ converges itself to $\left( u,v\right) .$ Now,
letting $n\rightarrow +\infty $ in (2.15), we conclude that $\left(
u,v\right) $ is a positive weak solution of system (\ref{1.1}).
\end{proof}

\bigskip

\begin{acknowledgement}
For any decision, the authors are grateful to the anonymous referees for the
careful reading and their important observations/suggestions for sake of
improving this paper. In memory of Mr. Mahmoud ben Mouha Boulaaras
(1910--1999).
\end{acknowledgement}

\ \ \ \ \ \ \ \ \ \ \ \ \ \

\end{document}